\documentclass[letterpaper, 10 pt, conference]{ieeeconf}  

\IEEEoverridecommandlockouts                              
\usepackage{graphicx}      
\usepackage{xcolor}
\usepackage{amsmath,amssymb,mathtools}
\usepackage{soul}
\usepackage{tikz}
\usetikzlibrary{positioning}
\usepackage{algorithm}
\usepackage{algpseudocode}
\usepackage[noadjust]{cite}

\newtheorem{thm}{Theorem}
\newtheorem{cor}{Corollary}

\newtheorem{assum}{Assumption}
\newtheorem{prop}{Proposition}

\newtheorem{defn}{Definition}

\newtheorem{rem}{Remark}

\renewcommand{\Re}{\mathbb{R}}
\newcommand{\calA}{\mathcal{A}}
\newcommand{\calB}{\mathcal{B}}
\newcommand{\calP}{\mathcal{P}}
\newcommand{\calX}{\mathcal{X}}
\newcommand{\calY}{\mathcal{Y}}
\newcommand{\nxn}{{n\times n}}
\newcommand{\Ne}{\mathbb{N}}
\newcommand{\calD}{\mathcal{D}}
\newcommand{\SPn}{\Re^\nxn_{\succ0}}
\newcommand{\Prob}{\mathbb{P}}
\newcommand{\Expec}{\mathbb{E}}

\newcommand{\rhohat}{\hat{\rho}}
\newcommand{\Sph}{\mathbb{S}}
\newcommand{\lambdamin}{\lambda_{\mathrm{min}}}
\newcommand{\spann}{\mathrm{span}}
\DeclareMathOperator*{\argmin}{arg\,min}

\newcommand{\InputBlock}[1]{\State \textbf{Input:}\par\addvspace{2pt} \parbox[t]{\linewidth}{#1}\par\addvspace{6pt}}
\newcommand{\OutputBlock}[1]{\State \textbf{Output:} #1}

\title{\LARGE \bf
A Stochastic-Optimization-Based Adaptive-Sampling Scheme for Data-Driven Stability Analysis of Switched Linear Systems
}

\author{Alexis Vuille, Guillaume O.~Berger and Rapha\"el M.~Jungers
\thanks{This project has received funding from the European Research Council (ERC) under the European Union's Horizon 2020 research and innovation programme under grant agreement No 864017 -- L2C.}
\thanks{GB is a FNRS Postdoctoral Researcher, RJ is a FNRS honorary Research Associate. We are from UCLouvain, Louvain-la-Neuve, Belgium (e-mail:
firstname.lastname@uclouvain.be).}%
}

\begin{document}

\maketitle
\thispagestyle{empty}
\pagestyle{empty}
\begin{abstract}                
We introduce a novel approach based on stochastic optimization to find the optimal sampling distribution for the data-driven stability analysis of switched linear systems. 
Our goal is to address limitations of existing approaches, in particular, the fact that these methods suffer from ill-conditioning of the optimal Lyapunov function, which was shown in recent work to be a direct consequence of the way the data is collected by sampling uniformly the state space.
In this work, we formalize the notion of optimal sampling distribution, using the perspective of stochastic optimization.
This allows us to leverage tools from stochastic optimization to estimate the optimal sampling distribution, and then use it to collect samples for data-driven stability analysis of the system.
We show in numerical experiments (on challenging systems of dimension up to five) that the overall procedure is highly favorable in terms of data usage compared to existing methods using fixed sampling distributions.
Finally, we introduce a heuristic that combines data points from previous samples, and show empirically that this allows an additional substantial reduction in the number of samples required to achieve the same stability guarantees.
\end{abstract}


\section{Introduction}\label{sec:introduction}

In recent years, data-driven methods have gained a lot of attention for the study of cyber-physical systems because of the increasing number of applications in which no model of the system is available. 
At the same time, data has become more and more accessible due to the outbreak of cheap, accurate sensors, user feedback and open-source databases.
Finally, statistical learning, the mathematical field of learning from data, has known many great advances in recent years both in theory and practice \cite{shalevshwartz2014understanding}.
All this together opened the door to a new era in control theory where control and system analysis is made from data harvested from observation of the system and comes with formal guarantees of correctness; we refer the reader to \cite[Chapter~11]{mitra2021verifying} for an introduction and further references on data-driven verification and control of cyber-physical systems.

In this paper, we consider a prototypical class of cyber-physical systems, known as \emph{switched linear systems} \cite{sun2011stability}.
These systems consist of several linear modes among which the system can switch over time.
They appear naturally in a wide range of applications \cite{liberzon2003switching}, or as approximations of more complex systems.
A crucial question in the study of switched linear systems is their stability analysis \cite{sun2011stability}, which turns out to be a very challenging problem in general even when the model of the system is available \cite{jungers2009thejoint}.
For instance, approximating the rate of convergence of the system (known as the Joint Spectral Radius, or JSR) is known to be NP-hard \cite{jungers2009thejoint}.
Nevertheless, several approximation techniques have been proposed in the last decades leading to good results in the model-based setting \cite{jungers2009thejoint,ahmadi2014joint}.

However, a model of the system is not always available.
Therefore, several approaches were proposed in recent years for the data-driven analysis of the JSR of switched linear systems \cite{kenanian2019data,berger2021chanceconstrained,rubbens2021datadriven,wang2021datadriven,vuille2024data}.
These approaches use advanced tools from statistical learning, such as scenario optimization \cite{calafiore2006thescenario}, to learn a Lyapunov function for the system to derive bounds on the JSR and provide probabilistic guarantees on the correctness of the bound.

However, the classical approaches \cite{kenanian2019data,berger2021chanceconstrained,rubbens2021datadriven} strongly suffer from a bad choice of the distribution used to sample the data.
This was shown in \cite{vuille2024data}---where additionally an intuitive approach for finding a better sampling distribution was proposed, demonstrating huge gain in data usage.
In this paper, we leverage these promising preliminary results in two complementary directions.
First, we formalize the notion of \emph{best} sampling distribution, using the perspective of stochastic optimization.
Namely, we search for the sampling distribution that gives \emph{in average} the best conditioning of the learned Lyapunov function (shown in \cite{vuille2024data} to be one of the main sources of conservatism).
Second, we use stochastic optimization techniques (namely stochastic gradient descent) to learn this best sampling distribution form data.
We show with numerical experiments that the overall procedure allows us to certify stability using less samples than other fixed-distribution data-driven methods \cite{berger2021chanceconstrained}.
Finally, we introduce a heuristic that combines data points from previous samples, empirically demonstrating a substantial reduction in the number of samples required to achieve the same stability guarantees, and improving upon the heuristic method of \cite{vuille2024data} in all experiments.


\subsection*{Related Works}

Data-driven stability analysis of switched linear systems is studied in \cite{kenanian2019data,berger2021chanceconstrained,rubbens2021datadriven,wang2021datadriven}.
These approaches suffer from the curse of dimensionality, strongly amplified by the fact that the data points are sampled uniformly, as demonstrated in our recent work \cite{vuille2024data}.
In \cite{vuille2024data}, we introduced the method of adaptive sampling to alleviate the dependency on the sampling, showing a reduction of sample complexity of several orders of magnitude.
Yet, the adaptive sampling methodology in \cite{vuille2024data} is strongly heuristic.
It consists in two phases.
1) Sample points uniformly, and learn a Lyapunov function from these samples.
2) Use this Lyapunov function to obtain a new sampling distribution; sample from this distribution and derive the stability guarantees.
Although simple, this strategy showed great potential.
However, key questions such as ``what is actually a good sampling distribution?'', or ``can we do better by doing more than two steps (each time using the sampling distribution obtained at the previous step)?''.
This work aims to address these questions by (i) formalizing the notion of optimal distribution for data-driven stability analysis of switched linear systems, from the lens of stochastic optimization, and (ii) providing numerical techniques to compute the optimal distribution.
We demonstrate in numerical examples the strong benefits of the proposed solution over the previous approaches.

\paragraph*{Notation}
$\SPn$ denotes the set of positive definite $n\times n$ matrices. Given $P \in \SPn$, we denote the quadratic norm of $P$ as $\|x\|_P = \sqrt{x^TPx}$.
For $N\in\Ne$, $[N]$ denotes the set $\{1,\ldots,N\}$.
We denote the unit sphere in $\Re^n$ by $\Sph^{n-1}$.
Given two datasets of size $N$, $\mathcal{X}\coloneqq\{x_i\}_{i=1}^N$ and $\mathcal{X}'\coloneqq\{x_i'\}_{i=1}^N$, we denote their element-wise pairing as $\mathcal{X}\|\mathcal{X}'\coloneqq\{(x_i,x_i')\}_{i=1}^N$. 

\section{Problem Statement}\label{sec:problem_statement}

We consider a discrete-time \emph{switched linear system} with $m$ modes:
\begin{equation}\label{eq:sys}
x(t+1) \in \{Ax(t) : A\in\calA\},
\end{equation}
wherein $\calA=\{A_1,\ldots,A_m\}\subseteq\Re^\nxn$ is a set of $m$ matrices in $\Re^\nxn$.
Since $\calA$ characterizes the system in \eqref{eq:sys}, in the following, we will often refer to the system simply by $\calA$.
A \emph{trajectory} of $\calA$ is a function $x:\Ne\to\Re^n$ such that for all $t\in\Ne$, the condition in \eqref{eq:sys} holds.

We are interested in the stability of system \eqref{eq:sys}.
We remind that \eqref{eq:sys} is \emph{asymptotically stable} if all trajectories of $\calA$ converge to the origin.
The rate of exponential convergence is called the \emph{Joint Spectral Radius} (JSR) of $\calA$.

\begin{defn}\label{def:jsr-rate-stability}
The \emph{joint spectral radius} of $\calA$, denoted by $\rho(\calA)$, is the infimum of all $r\geq0$ for which there exists $C\geq1$ such that every trajectory $x$ of $\calA$ satisfies that for all $t\in\Ne$, $\lVert x(t)\rVert\leq Cr^t\lVert x(0)\rVert$.
\end{defn}

\subsection{Quadratic Approximation of the JSR}

The JSR is notoriously difficult to approximate, even when the matrices in $\calA$ are known \cite{jungers2009thejoint}.
One way to obtain an upper bound on the JSR is by finding a quadratic \textit{Lyapunov} function for the system. The contraction rate associated with this function then provides an upper bound.

\begin{defn}
Given a positive definite matrix $P\in\SPn$, we define the \emph{contraction rate} of $\calA$ with respect to $P$ by
\[
\rho(\calA,P)=\max_{x\in\Re^n\setminus\{0\},\,A \in\calA} \frac{\|Ax\|_P}{\|x\|_P}.
\]\vskip0pt
\end{defn}

The contraction rate is an upper bound on the JSR (see, e.g., \cite[Proposition~2.8]{jungers2009thejoint}):

\begin{thm}\label{thm:white-box-quadratic}
For any $P\in\SPn$, $\rho(\calA)\leq\rho(\calA,P)$.
\end{thm}

Hence, quadratic Lyapunov functions allow us to bound the JSR.
However, this approach, as other \emph{model-based} approaches for approximating the JSR, require the knowledge of $\calA$.
This is a limitation in several applications, thereby justifying the use of data-driven methods.

\subsection{Data-Driven Analysis and Random Data Collection}

In the setting considered in this paper (first introduced in \cite{kenanian2019data}), we collect data by setting the system to an initial state $x$ and observing the state $y$ after one time step.
Repeating this process $N$ times yields a data set comprising $N$ one-step trajectories $(x_i, y_i)\in\Re^n\times\Re^n$, where $y_i=Ax_i$ for some $A\in\calA$, for each $i\in[N]$.
We assume that the mode selection is a stochastic process, wherein each mode in $\calA$ has a nonzero probability of being applied independently at each sample:\footnote{Note that even though we consider the mode selection to be a stochastic process for data collection, the probabilistic guarantees we will obtain still hold for all possible trajectories of the system (even ones for which the mode selection is not a stochastic process). However, without assumption 1, that would be not the case as a mode could be never sampled in our dataset for all datasets of finite size. Similarly, while we consider in the rest of the paper different kinds of distributions for the initial states to collect the data and analyse the stability, the stability analysis and the confidence bounds remain valid for all possible trajectories and all possible initial states.}


\begin{assum}\label{assum:uniform-matrix}
There exists $\alpha\in(0,1]$ such that for all $A\in\calA$ and $i\in[N]$,
\[
\Prob\left[\, y_i=Ax_i \mid x_i,\:\{(x_j,y_j)\}_{j\neq i} \,\right] \geq \alpha.
\]\vskip0pt
\end{assum}
Regarding the choice of the initial state of each sample $(x_i,y_i)$, we assume that the initial state $x_i$ can be chosen for each $i\in[N]$.
In particular, we will choose $x_i$ randomly according to some distribution that we can design.
In particular, we will consider the standard Gaussian distribution (reminded below), possibly after applying a change of basis:%
\footnote{Because of the scaling invariance, a data point $(x_i,y_i)$ carries the same information as the data point $(\lambda x_i,\lambda y_i)$ for every $\lambda\neq0$.
This is why sampling with respect to the standard Gaussian distribution is equivalent in this problem to sampling with respect to the uniform distribution on the unit sphere $\Sph^{n-1}$ in $\Re^n$.}

\begin{defn}
A random variable $X$ with value in $\Re^n$ has \emph{standard Gaussian distribution} if its probability density function (pdf) $f$ satisfies for all $x\in\Re^n$, $f(x)\propto e^{-\frac12\lVert x\rVert^2}$.
\end{defn}

\begin{rem}
As we will see in Section~\ref{sec:adaptive_sampling}, the fact that the sampling distribution can be chosen is key in our framework since our approach to improve data-scalability is to learn an optimal sampling distribution.
This assumption is realistic in a wide range of applications, namely when one has access to the system has a (stochastic) input--output black-box.
\end{rem}



\subsection{Fixed Sampling Distribution}\label{sec:previous_bound}

The data-driven approach in \cite{kenanian2019data} (refined in \cite{berger2021chanceconstrained}) provides probabilistic upper bounds on the JSR from data collected from a fixed sampling distribution.
See also \cite{rubbens2021datadriven,wang2021datadriven} for similar data-driven approaches using a fixed sampling distribution.
We remind here the main result of \cite{berger2021chanceconstrained} because this will be useful for the rest of this paper.


The approach of \cite{kenanian2019data,berger2021chanceconstrained} works as follows.
Given a data set $\mathcal{D} = \{(x_i,y_i)\}_{i=1}^N$ consisting of $N$ one-step trajectories, we formulate the problem of finding a positive definite matrix $P$ with the smallest \emph{data-based contraction rate} defined by
\[
\rhohat(\calD,P)=\max_{i\in[N]} \frac{\|y_i\|_P}{\|x_i\|_P}.
\]
Hence, we aim to solve the optimization problem
\begin{equation}\label{prob:data-driven_problem}
\min_{P\in\calP} \, \rhohat(\calD,P).
\end{equation}
where $\calP$ is a closed subset of $\SPn$ (we assume that $I\in\calP$).
Note that \eqref{prob:data-driven_problem} can be solved efficiently, as it is a quasi-convex optimization problem \cite{berger2021chanceconstrained}.
The optimal cost of \eqref{prob:data-driven_problem} is denoted by $\gamma_\star(\calD)$, and the optimal solution, if it exists, by $P_\star(\calD)$.
We assume without loss of generality that if $P_\star(\calD)$ exists, then it is unique (this can be done by using a tie-breaking rule \cite{berger2021chanceconstrained}).
When $\calD$ is clear from the context, we write $\gamma_\star$ and $P_\star$ instead of $\gamma_\star(\calD)$ and $P_\star(\calD)$.

Under some mild assumption (Assumption \ref{assum:no-barabanov}) on $\mathcal{A}$, one can guarantee an upper bound on the JSR of $\calA$ with high confidence (Theorem \ref{thm:data-driven-bound-previous}).\footnote{This guarantees the non-degeneracy property to apply the PAC bounds from Scenario Optimization. See for example \cite{calafiore2006thescenario}.}
\begin{defn}
    A matrix is said to be \textit{Barabanov} if it is diagonalizable and all its eigenvalues have the same modulus.
\end{defn}
\begin{assum}\label{assum:no-barabanov}
The matrices in $\calA$ are not Barabanov.
\end{assum}

\begin{thm}[\cite{berger2021chanceconstrained}]\label{thm:data-driven-bound-previous}
Let $d$ be the dimension of $\spann(\calP)$, and $N\geq d$.
Let $\{x_i\}_{i=1}^N\subseteq\Re^n$ be sampled i.i.d.~following the standard Gaussian distribution.
Let Assumptions \ref{assum:uniform-matrix} and \ref{assum:no-barabanov} hold.
Let $\beta\in(0,1]$.
Then, with probability $1-\beta$ on the sampling of $\calD\coloneqq\{(x_i,y_i)\}_{i=1}^N$, it holds that\footnote{Note that the first inequality in \eqref{eq:bound-jsr} is always satisfied, while the second one is guaranteed to hold with probability at least $1-\beta$.}
\begin{align}
&\rho(\calA) \leq \rho(\calA,P_\star) \leq \gamma_\star \cdot f(\beta,\kappa(P_\star),N,d,\alpha,n), \label{eq:bound-jsr}
\end{align}
wherein
\begin{itemize}
    \item $\kappa(P)=\sqrt{\frac{\det(P)}{\lambdamin(P)^n}}$;
    \item $f(\beta,k,N,d,\alpha,n) = \frac{1}{\sqrt{1-I^{-1}\left(\frac{k}\alpha\Phi^{-1}(\beta;d-1;N);\frac{n-1}{2};\frac{1}{2}\right)}}$;
    \item $I^{-1}(y;a;b)$ is the inverse incomplete regularized beta function, i.e., it is the unique $x\in[0,1]$ such that
    \[
        I(x;a;b) \coloneqq \frac{\int_{0}^x t^{a-1} (1-t)^{b-1} \mathrm{d}t}{\int_{0}^1 t^{a-1} (1-t)^{b-1} \mathrm{d}t} = y;
    \]
    \item $\Phi^{-1}(\beta;\zeta;N)$ is the unique $\epsilon\in[0,1]$ such that
    \[
        \Phi(\epsilon;\zeta;N) \coloneqq \sum_{i=0}^\zeta \binom{N}{i} \epsilon^i (1-\epsilon)^{N-i} = \beta.
    \]
    \item $\alpha$ is defined in Assumption \ref{assum:uniform-matrix}.
\end{itemize}
\end{thm}

\begin{rem}
The \emph{inflation factor} $f(\beta,\kappa(P_\star),N,d,\alpha,n)$
is the source of the additional conservatism of the data-driven approach, compared to the model-based approach.
Unfortunately, this factor increases exponentially with the value of $\kappa(P_\star)$. Hence, it is crucial to make the value of $\kappa(P_\star)$ as small as possible.
This is the purpose of the adaptive sampling approach, described next.
\end{rem}

\subsection{Adaptive Sampling Distribution}

The two-step approach was proposed in \cite{vuille2024data} as an effective method to reduce the value of $\kappa(P_\star)$ through adaptive sampling.
The key idea behind this approach is that a change of basis of the state variable $x$ can change the value of $P_\star$ and thereby the value of $\kappa(P_\star)$. 
Hence, our approach aims to find the change of coordinates for which $\kappa(P_\star)$ is the smallest.

More precisely, the change of basis and sampling distribution works as follows.
Given an invertible matrix $B \in \mathbb{R}^{n \times n}$ and a data set $\calX = \{x_i\}_{i=1}^N$ sampled i.i.d.~from the standard Gaussian distribution, we define the data set in the basis $B$ by $\calX' \coloneqq \{x_i' \coloneqq B x_i\}_{i=1}^N$.
This transformed data set can then be fed to the system oracle, providing the data set $\calY' \coloneqq \{y_i'\}_{i=1}^N$ Finally, we can apply the reverse change of basis to obtain the data set $\calY \coloneqq \{y_i \coloneqq B^{-1} y_i'\}_{i=1}^N$. This sampling process is illustrated in Figure \ref{fig:adaptive-sampling}.
We denote the resulting datasets by $\mathcal{D}=\calX\|\calY$ and $\mathcal{D}'= \calX'\|\calY'$.
The interest of the change of basis is that if $B$ is chosen appropriately, then $\kappa(P_\star(\calD))$ can be expected to be close to one.
In fact, it is shown in \cite[Proposition 8]{berger2021chanceconstrained} that in the model-based setting (i.e., when $\calA$ is known and $N\to\infty$), one can choose $B$ so that $\kappa(P_\star(\calD))=1$ with probability one.
However, in the data-driven context, $\calA$ is unknown and we want to keep $N$ small (thus finite).
Therefore, \cite{vuille2024data} proposed a two-step approach, consisting in first guessing a change of basis $B$ by using a initial dataset $\calD_\circ$ (of size $N_\circ)$, and then using this $B$ to build the dataset $\calD$ (of size $N$), and find $\gamma_\star(\calD)$ and $P_\star(\calD)$.
This results in an approach that requires (empirically) much less data ($N_\circ+N$), compared to fixed-sampling distribution approaches, to provide stability guarantees.

\begin{figure}
    \centering
    \includegraphics[width=\linewidth]{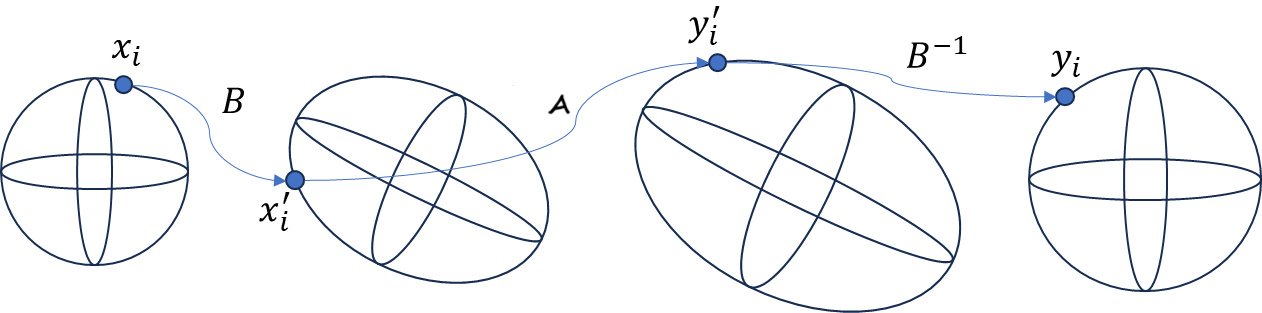}
    \caption{Sampling with change of basis $B$.}
    \label{fig:adaptive-sampling}
     \vspace{-5mm}
\end{figure}

Yet, despite promising empirical results, the construction of the change of basis $B$ in \cite{vuille2024data} is essentially intuitive.
Also, the notion of ``good'' change of basis is not clearly defined.
We address these challenges in Section \ref{sec:adaptive_sampling} below.
Then, in Section \ref{sec:heuristic}, we build upon the results in Section \ref{sec:adaptive_sampling} to provide a heuristic aimed at improving the data usage even further, as demonstrated empirically on numerical examples.

\section{Adaptive Sampling Through the Lens of Stochastic Optimization}\label{sec:adaptive_sampling}

Building on the observations made in the previous section, we seek a change of basis $B$ for which $\kappa(P_\star(\calD))$ is expected to be small when $\calD$ is built as explained in the previous section.
This will lead to the definition of the optimal $B$ as the solution of a stochastic optimization problem.
Building on this, we will then propose an numerical method to find the optimal change of basis.


Given $B\in\Re^\nxn$ invertible, we denote the distribution of the datasets $\calD=\calX\|\calY$ and $\calD'= \calX'\|\calY'$, built as explained in Section \ref{sec:problem_statement}, by $D^N(B)$ and $D'^N(B)$ respectively.





\subsection{Stochastic Optimization Problem}

We formalize the property of being the optimal change of basis $B$, which captures the fact that the \emph{expected} value of $\log(\kappa(P_{\star}(\calD)))$ will be minimized among all $B$ in some subset $\calB$ of invertible $\nxn$ matrices\footnote{In our experiments, we consider $\mathcal{B}$ to be a compact subset of positive symmetric definite matrice of the from $\mathcal{B} \coloneqq \{B \in \SPn: I \preceq B = B^T \preceq \alpha I\}$, for some $\alpha > 1$, which guarantees the existence of a solution.}:
\begin{equation}\label{eq:so-basic}
    \min_{B \in \calB}\, \Expec_{\calD \sim D^N(B)}[\log(\kappa(P_{\star}(\calD)))].
\end{equation}

We consider the logarithm as it leads in the end to a more numerically stable method to solve the problem.
We now introduce a closely related problem, which will be used for optimization:
\begin{equation}\label{eq:so-special}
    \min_{B \in \calB}\, \Expec_{\calD' \sim D'^N(B)} [\log(\kappa(B^{\top} P_{\star}(\calD)B))]
\end{equation}
It turns out that solving \eqref{eq:so-basic} or \eqref{eq:so-special} is equivalent:

\begin{thm}
Problems \eqref{eq:so-basic} and \eqref{eq:so-special} have the same set of optimal solutions.
\end{thm}

\begin{proof}
Observe that
\begin{align*} 
    P_{\star}(\calX\|\calY) &= \argmin_P \max_{x_i\in\calX,y_i\in\calY} \frac{\|y_i\|_P}{\|x_i\|_P} \\
    &= \argmin_P \max_{x_i'\in\calX',y_i'\in\calY'} \frac{\|B^{-1}y_i'\|_P}{\|B^{-1}x_i'\|_P}\\
    &= B^{\top} \argmin_{P'} \max_{x_i'\in\calX',y_i'\in\calY'} \frac{\|y_i'\|_{P'}}{\|x_i'\|_{P'}} B\\
    &= B^{\top} P_{\star}(\calX'\|\calY') B,
\end{align*}
where the second equality come from the definition of the sets $\calX,\calX',\calY,\calY'$ and the third equality is obtained by doing a change of variable $P' \coloneqq B^{-\top} P B^{-1}$.
\end{proof}

\subsection{Stochastic Gradient Descent}

We solve \eqref{eq:so-special} by employing a stochastic gradient algorithm.
Since the distribution involved in the objective function of \eqref{eq:so-special} depends on the decision variable $B$, we use a stochastic gradient algorithm accounting for \emph{decision-dependent distribution}; see, e.g., \cite{perdomo2020performative}.

Concretely, the method works as follows: given an estimate of the optimal sampling distribution $B_k$, we sample a dataset $\calD_k'$ according to $D'^N(B_k)$ and use this dataset to compute the gradient direction as
\[
g_k \coloneqq \nabla_B \log(\kappa(B^{\top} P_{\star}(\calD_k') B)).
\]
We then update the sampling distribution as $B_{k+1}' = B_k - \eta_k g_k$, for some predefined step size $\eta_k>0$.
If needed, we project on $B_{k+1}'$ on $\calB$, giving $B_{k+1}$.
As a first estimate of the optimal sampling distribution, we use $B_0 \coloneqq I$, which corresponds to no change of basis.

Using the formula of $\kappa(P)$, we can derive an explicit formula for the gradient\footnote{For some values of $P^{\top}BP$ (with Lebesgue zero measure) it's is only a subdiferential; see Appendix A.}:
\begin{equation}
    \nabla_B \log(\kappa(B^{\top} P B)) = B^{-T} - \frac{n PBvv^\top}{\lambda_{\mathrm{min}}(B^{\top} P B)},
\end{equation}
where $v \coloneqq v_{\mathrm{min}}(B^{\top} P_k B)^{\top}$ is the eigenvector associated to $\lambda_{\mathrm{min}}(B^{\top} P_k B)$, the minimum eigenvalue of $B^{\top} P_k B$.
The proof can be found in the Appendix A.

After $T$ steps of the stochastic gradient descent, we have our final estimate $B_T$ of the optimal sampling distribution.
We use it to sample a dataset $\calD_T$ according to $D^N(B_T)$, and we compute the probabilistic upper bound as $\gamma_\star(\calD_T) \cdot f(\beta,\kappa(P_\star(\calD_T)),N,d,\alpha,n)$.
This gives Algorithm \ref{alg:SGD}.
\begin{algorithm}
\caption{Stochastic Optimization Upper Bound}
\label{alg:SGD}
\begin{algorithmic}[1]
\InputBlock{
    - A black-box switched linear system $\mathcal{A}$.\\
    - $T$: number of iterations.\\
    - $(\eta_k)_{k=0}^{T-1}$: step sizes.\\
    - $N_{batch}$: batch size.
}
\State Set $B_0 \coloneqq I$
\For{$k = 0$ to $T-1$}
    \State Sample $\calD_k' \coloneqq \{(x_i',y_i')\}_{i=1}^{N_{batch}}\sim D'^{N_{batch}}(B_k)$.
    \State Using $\calD_k'$, solve \eqref{prob:data-driven_problem} to compute $P_k \coloneqq P_{\star}(\calD_k')$
    \State Compute the gradient: $g_k \coloneqq \nabla_B \log(\kappa(B^{\top} P_k B))$
    \State Update the estimate: $B_{k+1}' \coloneqq B_{k} - \eta_k g_k$
    \State Project $B_{k+1}'$ on $\calB$ to obtain $B_{k+1}$
\EndFor
\State Sample $\calD_T \coloneqq \{(x_i,y_i)\}_{i=1}^{N_{batch}}\sim D^{N_{batch}}(B_T)$.
\State Solve~\eqref{prob:data-driven_problem} to get $\gamma_\star(\calD_T)$ and $P_\star(\calD_T)$.
\OutputBlock{$\gamma_\star(\calD_T) \cdot f(\beta,\kappa(P_\star(\calD_T)),N_{batch},d,\alpha,n)$}
\end{algorithmic}
\end{algorithm}

It turns out that if $\kappa(P_\star(\calD_T))$ is expected to be close to one (which can be estimated from the objective value of Algorithm \ref{alg:SGD}), then it can be beneficial to use $\calP=\{I\}$ (i.e., fix $P=I$), because the small increase of $\gamma_\star$ (since $\calP$ is more restricted) will be compensated by the fact that $f(\beta,k,N,d,\alpha,n)$ is smaller when $d=1$:

\begin{prop}\label{prop:fixing-P}
Let $\calD=\{(x_i,y_i)\}_{i=1}^N$.
It holds that $\gamma_\star\leq\rhohat(\calD,I)\leq\gamma_\star\kappa(P_\star)$.
\end{prop}

\begin{proof}
The first inequality is direct from \eqref{prob:data-driven_problem}.
To prove the second inequality, first assume without loss of generality that $\lambdamin(P_\star)=1$.
Denote $k=\kappa(P_\star)$.
Then, observe that $\det(P_\star)\leq k^2$, which implies that $I\preceq P_\star\preceq k^2I$.
Hence, we get that for each $i\in[N]$,
\[
\lVert y_i\rVert^2\leq y_i^\top P_\star y_i^{} \leq \gamma_\star^2 x_i^\top P_\star x_i^{} \leq \gamma_\star^2k^2 \lVert x_i\rVert.
\]
This shows that $\rhohat(\calD,I)\leq k\gamma_\star$.
\end{proof}

\begin{cor}\label{cor:fixing-P}
Let $B\in\Re^\nxn$ be invertible.
For any $\epsilon>1$, it holds with probability $1-\delta/\log(\epsilon)$ on $\calD\sim D^N(B)$ that $\rhohat(\calD,I)\leq\epsilon\gamma_\star(\calD)$, where $\Expec_{\calD \sim D^N(B)}[\log(\kappa(P_{\star}(\calD)))]=\delta$.
\end{cor}

\begin{proof}
Apply Markov's inequality on $\log(\frac{\rhohat(\calD,I)}{\gamma_\star(\calD)})$ which is always nonnegative by Proposition \ref{prop:fixing-P}, and its expectation is smaller than or equal to $\delta$ by Proposition \ref{prop:fixing-P}.
\end{proof}



\subsection{Open Questions on Convergence and Optimality}

\subsubsection{Uniqueness of the solution}

Experimentally, \eqref{eq:so-basic} and \eqref{eq:so-special} seem to admit only one local optimum (up to a nonzero scaling) which is the global optimum.
Indeed, one can see on the Figure~\ref{fig:comparison} that the data-driven JSR converges the toward model-based one.
However, a formal proof of this claim would be valuable. 

\subsubsection{Convergence}

Furthermore, experimental observations indicate that stochastic gradient descent seem to converge, provided the step sizes satisfy the conditions $\sum_{k=0}^{\infty} \eta_k = +\infty$ and $\lim_{k \rightarrow +\infty} \eta_k = 0$.

Establishing a theorem that guarantees convergence under specific conditions would be highly desirable. The main challenge for this lies in the fact that the probability distribution involved in the expectation depends on the decision variable. There exist some works that obtain such guarantees but they require the objective to be strongly convex \cite{perdomo2020performative}, which is not our case.
We will address these questions in the future.

\section{Reusing Previous Samples---a Heuristic Approach}\label{sec:heuristic}

The stochastic gradient descent algorithm has a key practical limitation in that each iteration uses only the new samples without exploiting the previous ones.
This can result in a prohibitively high overall number of samples.
This was already observed in previous work, leading to extended notions of the SGD, such as \emph{multi-pass} SGD \cite{lei2021generalization}.

In this section, we propose a modification of Algorithm \ref{alg:SGD} that reuses previous samples, and is ad-hoc to our problem.
Although not theoretically grounded, we show experimentally that this modified algorithm outperforms Algorithm \ref{alg:SGD} (and multi-pass SGD\footnote{The reason for our algorithm to outperform multi-pass SGD can be that the distribution is decision-dependent, thereby making the gradient computed from data collected with a different distribution less relevant.}) in terms of number of data needed to provide stability guarantees.\footnote{By heuristic method we mean that the computed change of basis has no whatsoever guarantee to converge to an optimal change of basis but nevertheless confidence bound from Theorem 1 still apply as it holds for any change of basis used.}

The two key differences of the modified algorithm are that (i) all previous and current samples are used in $\calD_k$ to compute $\gamma_\star(\calD_k)$ and $P_\star(\calD_k)$, and (ii) instead of moving in the direction of the gradient of the objective function with respect to $B$, we move in the direction of $P_\star(\calD_k)^{-1/2}$, which corresponds to the minimizer of $\log(\kappa(B^\top P_\star(\calD_k)B)$.
This approach was obtained as a heuristic, testing different options for choosing the update direction.



More precisely, our approach is as follows.
Starting with $B_0 = I$, we sample an initial dataset $\calD_0 \coloneqq \{(x_i',y_i')\}_{i=1}^{N_0}$ according to $D'^{N_0}(B_0)$.
Then, at each step $k$, given $\calD_k$ and $B_k$, we update $B$ in the direction of $P_{\star}(\calD_k)^{-1/2}$:
\[
B_{k+1} \coloneqq (1 - \eta_k) B_k + \eta_k P_{\star}(\calD_k)^{-1/2}.
\]
Next, we draw a new sample $\{(x_k',y_k')\}$ from $D'(B_{k+1})$ and augment the dataset: $\calD_{k+1} \coloneqq \calD_k \cup \{(x_k',y_k')\}$.
This process continues until a predefined convergence criterion is satisfied.
In practice, we consider the algorithm converged when the average difference between successive iterations falls below a predefined threshold $\epsilon>0$, or when a maximum number of iterations is reached.
This is implemented in Algorithm \ref{alg:HEUR}.

\begin{algorithm}
\caption{Heuristic Optimization Upper Bound}
\label{alg:HEUR}
\begin{algorithmic}[1]
\InputBlock{
    - A black-box switched linear system $\mathcal{A}$.\\
    - $N_0$: initial sample size.\\
    - $N$: total sample budget.\\
    - $T$: number of iterations.\\
    - $(\eta_k)_{k=0}^{T-1} \subseteq [0,1]$: step sizes.\\
    - $\epsilon > 0$, $K\in\Ne$: convergence criteria parameters.
}
\State Initialize $B_0 \coloneqq I$.
\State Sample $\calD_0' \coloneqq \{(x_i',y_i')\}_{i=1}^{N_0}\sim D'^{N_0}(B_0)$.
\For{$k = 0$ to $T-1$}
    \State Solve \eqref{prob:data-driven_problem} using $\calD_k'$ to compute $P_k \coloneqq P_{\star}(\calD_k')$.
    \State Compute the direction as: $g_k \coloneqq B_k^{} - P_k^{-1/2}$.
    \State Update: $B_{k+1} \coloneqq B_k - \eta_k g_k$.
    \State \textbf{if} $\sum_{j=k-K}^{k} \|B_{j+1}-B_j\| \leq \epsilon$, terminate loop.
    \State Sample $\{(x_k',y_k')\}\sim D'^{1}(B_{k+1})$
    \State Augment the dataset: $\calD_{k+1} = \calD_k \cup \{(x_k',y_k')\}$.
\EndFor
\State Sample $\calD_T\coloneqq\{(x_i,y_i)\}_{i=1}^{N-\lvert\calD_T\rvert}\sim D^{N-\lvert\calD_T\rvert}(B_T)$.
\State Solve~\eqref{prob:data-driven_problem} to get $\gamma_\star(\calD_T)$ and $P_\star(\calD_T)$.
\OutputBlock{$\gamma_\star(\calD_T) \cdot f(\beta,\kappa(P_\star(\calD_T)),N-\lvert\calD_T\rvert,d,\alpha,n)$}
\end{algorithmic}
\end{algorithm}


\section{Experimental Results}

We demonstrate the effectiveness of our methods (Algorithms \ref{alg:SGD} and \ref{alg:HEUR}) on a synthetic example and on a consensus problem.
We also compare them with 1) the approach without adaptive sampling from \cite{berger2021chanceconstrained}, and 2) the state-of-the-art resampling technique from \cite{vuille2024data}.

\subsection{Synthetic Example}

We applied the four data-driven approaches on a randomly generated system of dimension $3$ with three modes.
To obtain statistics, we averaged the results over 25 experiences.
For the two-step approach from \cite{vuille2024data}, we used the heuristic from this paper for dataset splitting and parameters $\delta_1=10^2$, $\delta_2=1$.
For Algorithm \ref{alg:SGD}, we used $N_{\text{batch}}=200$, $T = \lfloor N / N_{\text{batch}} \rfloor$, and $\eta_k = 0.3 / (k+1)$.
For Algorithm \ref{alg:HEUR}, we used $T = \lfloor N / 2 \rfloor$, $\eta_k=0.3$, with convergence criterion parameters $\epsilon = 10^{-4}$ and $K=10$.
We considered $\mathcal{P} \coloneqq \{P \in \SPn: I \preceq P \preceq 10^2I\}$ and for Algorithm \ref{alg:SGD}, $\mathcal{B} \coloneqq \{B \in \SPn: I \preceq B = B^T \preceq 10^2I\}$.

The results---depicted in Figure \ref{fig:comparison}---clearly showcase the advantages of the adaptive sampling approach, leading to significantly better guarantees compared to the basic non-adaptive sampling method \cite{berger2021chanceconstrained}.
While Algorithm \ref{alg:SGD} outperforms the non-adaptive method, its performance remains suboptimal compared to the other adaptive sampling techniques (\cite{vuille2024data} and Algorithm \ref{alg:HEUR}).
However, Algorithm \ref{alg:HEUR}, which reuses all previous samples in a heuristic manner, achieves strong performance, surpassing the state-of-the-art method \cite{vuille2024data}.

\begin{figure}
    \centering
    \includegraphics[width=\linewidth,trim={25mm 3mm 35mm 15mm},clip]{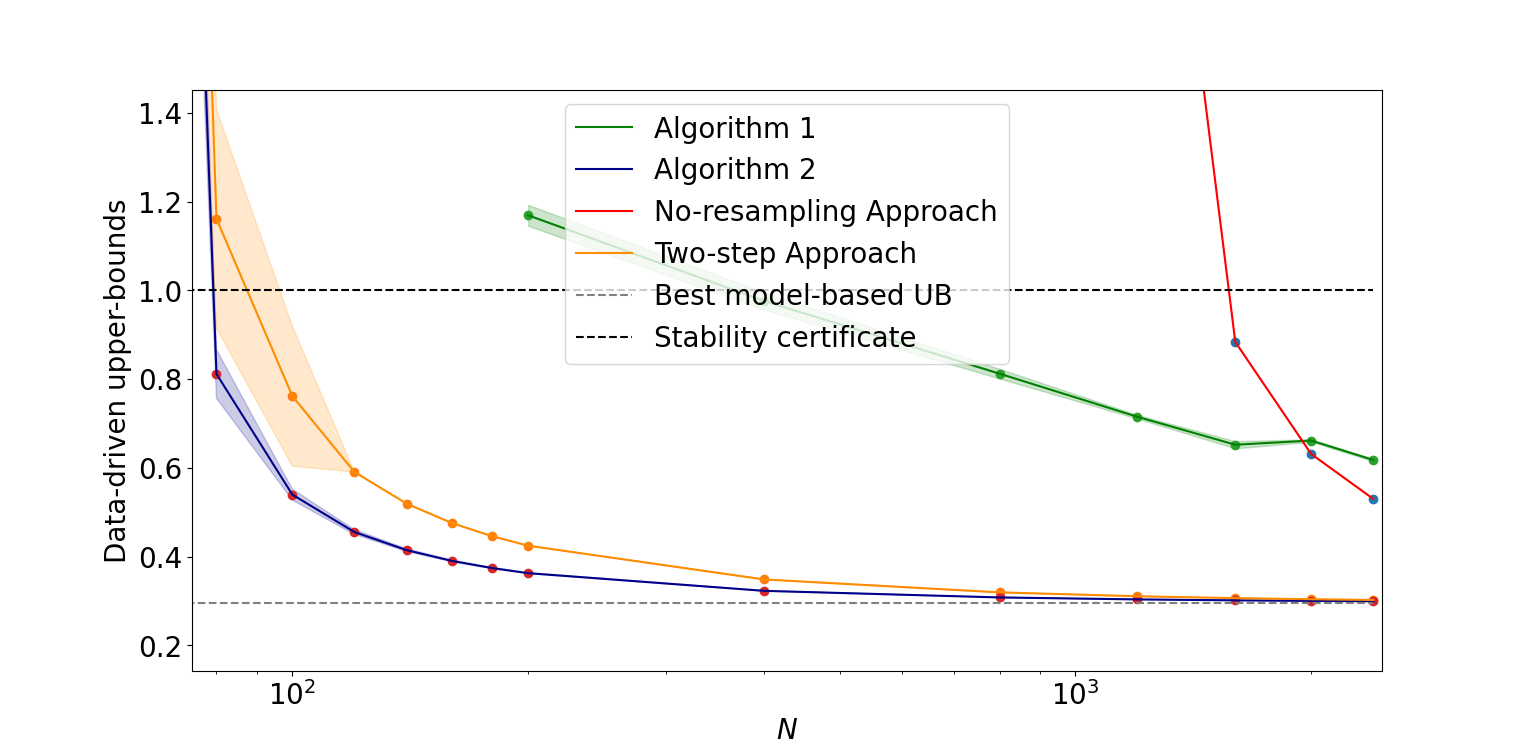}
    \caption{Comparison of the upper bounds provided by the various algorithms, as a function of the total number of samples $N$, with $\lvert\calA\rvert=3$, $\alpha=\frac{1}{|\calA|}$, $n=3$, and $\beta = 5\%$.
    Results are averaged over 25 experiences on one randomly generated system $\mathcal{A}$. The shaded area show the mean $\pm 1$ standard deviation.}
    \label{fig:comparison}
\end{figure}

\subsection{Consensus Network}
\label{sec:consensus}

We consider the problem of consensus in a hidden switching interaction network, illustrated in Figure \ref{fig:network}.
This problem can be equivalently formulated as a linear switched system in dimension $n=5$, and determining whether consensus is achieved translates to analyzing the stability of this switched linear system \cite{jadbabaie2003coordination}. Since the network is hidden and its model is unknown, we need to resort to data-driven methods to verify stability.

\begin{figure}
    \centering
    \includegraphics[width=\linewidth,trim={5mm 10mm 5mm 5mm},clip]{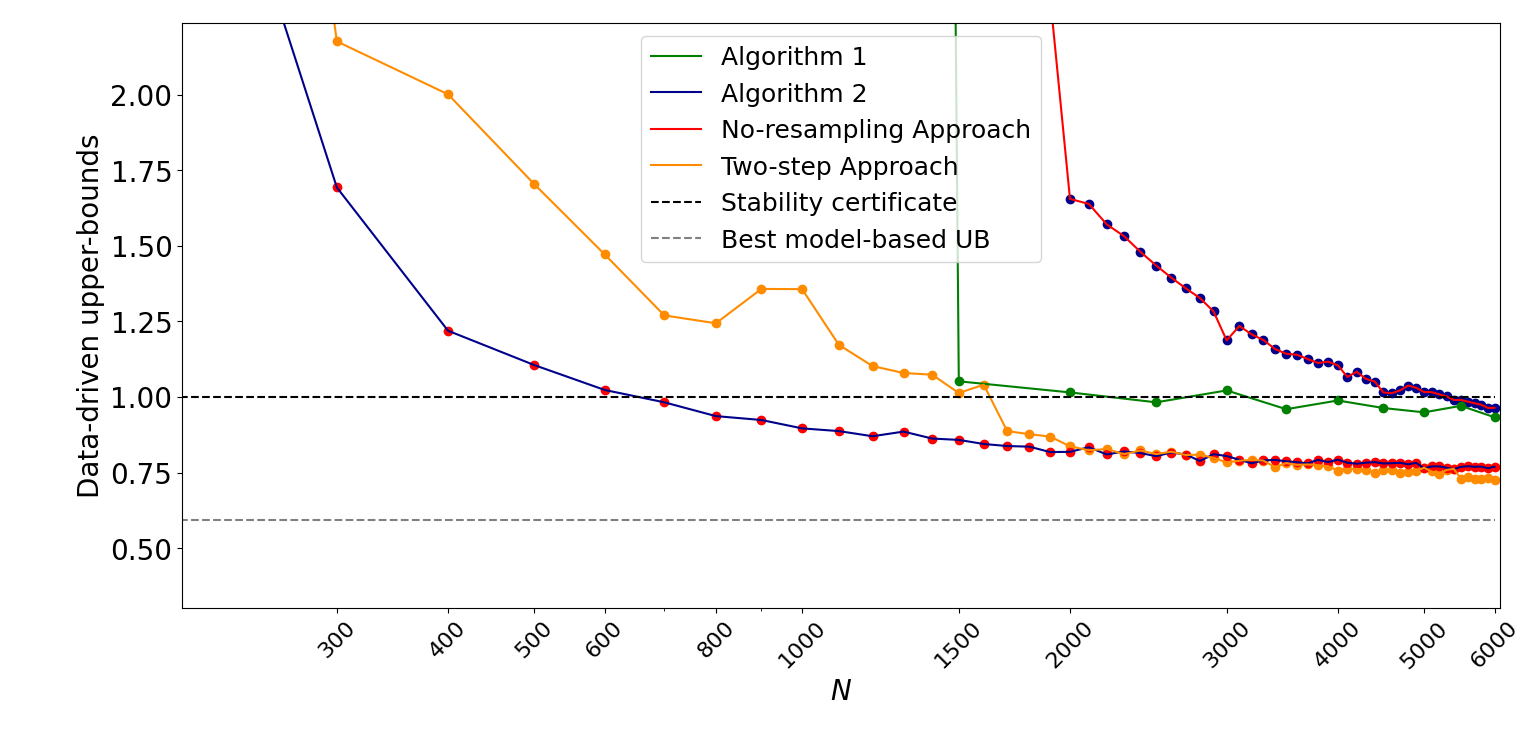}
    \caption{Data-driven upper bounds with confidence $1-\beta = 95\%$ on the JSR of the consensus network as a function of the total number of samples used $N$.}
    \label{fig:consensus}
\end{figure}

For the two-step approach from \cite{vuille2024data}, we used the dataset-splitting heuristic and parameters $\delta_1=10^3$, $\delta_2=1$.
For Algorithm \ref{alg:SGD}, we used $N_{\text{batch}} = 500$, $T = \lfloor N / N_{\text{batch}} \rfloor$, and $\eta_k = 0.3 / (k+1)$.
For Algorithm \ref{alg:HEUR}, we used $T = \lfloor N / 2 \rfloor$, $\eta_k=0.3$, $\epsilon = 10^{-4}$ and $K=10$.
We consider $\mathcal{P} \coloneqq \{P \in \SPn: I \preceq P \preceq 10^3I\}$ and for Algorithm \ref{alg:SGD}, $\mathcal{B} \coloneqq \{B \in \SPn: I \preceq B = B^T \preceq 10^3I\}$.

Figure \ref{fig:consensus} illustrates the total number of data points required to certify consensus with a confidence level of $1-\beta=95\%$ using the four data-driven methods.
The results highlight the substantial benefits of adaptive sampling approach over the basic no-resampling approach, as well as the superiority of the heuristic method (Algorithm \ref{alg:HEUR}) over the two-step approach.
Without resampling, 5400 data points were necessary to certify stability. By employing the stochastic optimization approach, this requirement is reduced to 2200 data points.
The two-step approach from \cite{vuille2024data} further improves this to 1600 data points.
However, Algorithm \ref{alg:HEUR} achieves a breakthrough, requiring only 600 data points to ensure the system's stability with 95\% confidence.

\begin{figure}
    \centering
    \vspace{2mm}
    \begin{tikzpicture}[scale=0.6]
        \node[draw, circle, fill=blue, inner sep=2pt] (1) at (0,2) {};
        \node[draw, circle, fill=blue, inner sep=2pt] (2) at (0,0) {};
        \node[draw, circle, fill=blue, inner sep=2pt] (3) at (0.5,1) {};
        \node[draw, circle, fill=blue, inner sep=2pt] (4) at (3,0) {};
        \node[draw, circle, fill=blue, inner sep=2pt] (5) at (3,2) {};
        \node[draw, circle, fill=blue, inner sep=2pt] (6) at (2.5,1) {};
        \draw[->,>=stealth] (4) -- (1);
        \draw[->,>=stealth] (5) -- (2);
        \draw[->,>=stealth] (1) -- (3);
        \draw[->,>=stealth] (5) -- (3);
        \draw[->,>=stealth] (5) -- (4);
        \draw[->,>=stealth] (6) -- (4);
        \draw[->,>=stealth] (6) -- (5);
        \draw[->,>=stealth] (2) -- (6);
        \draw[->,>=stealth] (3) -- (6);
    \end{tikzpicture}
    \hfill
    \begin{tikzpicture}[scale=0.6]
        \node[draw, circle, fill=blue, inner sep=2pt] (1) at (0,2) {};
        \node[draw, circle, fill=blue, inner sep=2pt] (2) at (0,0) {};
        \node[draw, circle, fill=blue, inner sep=2pt] (3) at (0.5,1) {};
        \node[draw, circle, fill=blue, inner sep=2pt] (4) at (3,0) {};
        \node[draw, circle, fill=blue, inner sep=2pt] (5) at (3,2) {};
        \node[draw, circle, fill=blue, inner sep=2pt] (6) at (2.5,1) {};
        \draw[->,>=stealth] (3) -- (1);
        \draw[->,>=stealth] (4) -- (1);
        \draw[->,>=stealth] (6) -- (1);
        \draw[->,>=stealth] (3) -- (2);
        \draw[->,>=stealth] (5) -- (2);
        \draw[->,>=stealth] (6) -- (2);
        \draw[->,>=stealth] (2) -- (3);
        \draw[->,>=stealth] (2) -- (4);
        \draw[->,>=stealth] (6) -- (4);
        \draw[->,>=stealth] (4) -- (5);
        \draw[->,>=stealth] (6) -- (5);
        \draw[->,>=stealth] (1) -- (6);
    \end{tikzpicture}
    \hfill
    \begin{tikzpicture}[scale=0.6]
        \node[draw, circle, fill=blue, inner sep=2pt] (1) at (0,2) {};
        \node[draw, circle, fill=blue, inner sep=2pt] (2) at (0,0) {};
        \node[draw, circle, fill=blue, inner sep=2pt] (3) at (0.5,1) {};
        \node[draw, circle, fill=blue, inner sep=2pt] (4) at (3,0) {};
        \node[draw, circle, fill=blue, inner sep=2pt] (5) at (3,2) {};
        \node[draw, circle, fill=blue, inner sep=2pt] (6) at (2.5,1) {};
        \draw[->,>=stealth] (3) -- (1);
        \draw[->,>=stealth] (4) -- (1);
        \draw[->,>=stealth] (5) -- (1);
        \draw[->,>=stealth] (6) -- (1);
        \draw[->,>=stealth] (6) -- (2);
        \draw[->,>=stealth] (6) -- (3);
        \draw[->,>=stealth] (6) -- (4);
        \draw[->,>=stealth] (1) -- (5);
        \draw[->,>=stealth] (2) -- (5);
        \draw[->,>=stealth] (4) -- (5);
        \draw[->,>=stealth] (6) -- (5);
        \draw[->,>=stealth] (1) -- (6);
        \draw[->,>=stealth] (5) -- (6);
    \end{tikzpicture}
    \caption{Three (unknown) interaction networks for the consensus problem in Sec. \ref{sec:consensus}.}
    \label{fig:network}
\end{figure}
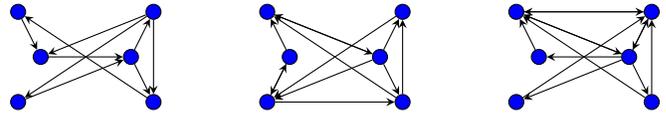



\bibliographystyle{IEEEtran}
\bibliography{myrefs}
\newpage
\section*{Appendix A}\label{sec:appendixA}

In this appendix, we compute the gradient \( \nabla_B f(B) \) for \( f(B) \coloneqq \log(\kappa(B^\top P B)) \).
As a reminder, $\kappa(A)$ is given by
\[
\kappa(A) = \left( \frac{\det(A)}{\lambdamin(A)^n} \right)^{1/2}
\]

Therefore,

\[
\begin{aligned}
f(B) &= \log\left( \left( \frac{\det(B^\top P B)}{\lambdamin(B^\top P B)^n} \right)^{1/2}\right)\\
&= \frac{1}{2} \log \det(B^\top P B) - \frac{n}{2} \log \lambdamin(B^\top P B)
\end{aligned}
\]

\noindent\textbf{Step 1: Computation of $\nabla_B \log \det(B^\top P B)$.}
Assuming $P$ is normalized such that $\det(P) = 1$, we have
\[
\det(B^\top P B) = \det(B^\top B) \cdot \det(P) = \det(B^\top B)
\]
For $X\in\SPn$, the gradient $\nabla_X \log \det(X) = X^{-1}$ (see, e.g., \cite[Appendix A.4.1]{boyd2004convex}).
Notice that $B^\top B\in\SPn$ since $B$ is invertible.
Therefore, $$\nabla_{B^\top B} \log(\det(B^\top B)) = (B^\top B)^{-1}.$$
Now, using the chain rule:
\[
\begin{aligned}
d(\log \det(B^\top B)) &= \operatorname{Tr}((B^\top B)^{-1} d(B^\top B)) \\
&= \operatorname{Tr}(B^{-1} B^{-\top} (B^\top dB + (dB)^\top B) ) \\
&= \operatorname{Tr}(B^{-1} dB) + \operatorname{Tr}(B B^{-1} B^{-\top}(dB)^\top)\\
&= \operatorname{Tr}(B^{-\top} (dB)^\top) + \operatorname{Tr}(B^{-\top}(dB)^\top)\\
&= \operatorname{Tr}(2B^{-\top} (dB)^\top).
\end{aligned}
\]
Therefore, the gradient is $\nabla_B \log \det(B^\top B) = 2B^{-\top}$.

\medskip

\noindent\textbf{Step 2: Computation of $\nabla_B\log \lambdamin(B^\top P B)$.}
Let $A = B^\top P B$ and $\lambdamin(A)$ be its minimum eigenvalue with corresponding normalized eigenvector $v$, i.e., $A v = \lambdamin(A) v$ and $\lVert v\rVert=1$. We assume that $\lambdamin(A)$ is unique (i.e., a simple root of the characteristic polynomial)\footnote{Otherwise, $\lambdamin(A)$ is not differentiable and the expression with any normalized eigenvector $v$ associated to $\lambdamin(A)$ is a subdifferential (for a proof see \cite{hiriart1999clarke}).}.
We show that the differential of $\lambdamin(A)$ is
\[
d(\lambdamin(A)) = v^\top dA v.
\]
On one hand, we have
\[
d(A v) = (dA) v + A (dv).
\]
On the other hand, we also have
\[
d(Av) = d(\lambdamin(A) v) = (d \lambdamin(A)) v + \lambdamin(A) (dv).
\]
Since $\lVert v\rVert=1$,
\[
\frac{1}{2} d(\lVert v\rVert^2) = \frac{1}{2} d(1^2)=0 = v^\top (dv).
\]
Combining the three previous equations and pre-multiplying by $v^\top$ leads to
\[
v^\top ( (dA) v + A (dv)) = (d \lambdamin(A)) \underbrace{v^\top v}_{=1} + \lambdamin(A) \underbrace{v^\top (dv)}_{=0}.
\]
Furthermore,
\[
\begin{aligned}
    v^\top A (dv) &= v^\top A^\top (dv) = (Av) (dv) = (\lambdamin(A)v)^\top (dv) \\
    &= \lambdamin(A) \underbrace{v^\top (dv)}_{=0} = 0.
\end{aligned}
\]
Thus, we finally get
\begin{equation*}
    d(\lambdamin(A)) = v^\top (dA) v,
\end{equation*}

Now, using back the expression of $A = B^\top PB$:
\[
\begin{aligned}
&d(\lambdamin(B^\top P B)) = v^\top d(B^\top P B) \, v \\
&= \operatorname{Tr}(d(B^\top PB) \, v v^\top) \\
&= \operatorname{Tr}((B^\top P (dB) + (dB)^\top P B) v v^\top) \\
&= \operatorname{Tr}(B^\top P (dB) \, v v^\top) + \operatorname{Tr}((dB)^\top P B v v^\top) \\
&= \operatorname{Tr}((B^\top P (dB) \, v v^\top)^\top) + \operatorname{Tr}(P B v v^\top (dB)^\top) \\
&= \operatorname{Tr}(v v^\top (dB)^\top P^\top B) + \operatorname{Tr}(P B v v^\top (dB)^\top) \\
&= \operatorname{Tr}(P B v v^\top (dB)^\top) + \operatorname{Tr}(P B v v^\top (dB)^\top) \\
&= \operatorname{Tr}(2 P B v v^\top (dB)^\top)
\end{aligned}
\]
Hence,
\[
\nabla_B \lambdamin(B^\top P B) = 2 P B v v^\top
\]
and
\[
\nabla_B \log \lambdamin(B^\top P B) = \frac{2 P B v v^\top}{\lambdamin(B^\top P B)}
\]

\noindent\textbf{Step 3: Combining the results to obtain $\nabla_B \log f(B)$.}
\[
\nabla_B f(B) = \frac{1}{2}\nabla_B \log \det(B^\top P B) - \frac{n}{2} \nabla_B \log \lambdamin(B^\top P B).
\]
Substituting the gradients:
\[
\nabla_B f(B) = \frac{1}{2} \left( 2B^{-\top} - \frac{2n P B v v^\top}{\lambdamin(B^\top P B)} \right).
\]
Simplifying:
\[
\nabla_B f(B) = B^{-\top} - \frac{nP B v v^\top}{\lambdamin(B^\top P B)}.
\]
\end{document}